\documentclass[-ejs]{imsart}

\RequirePackage[OT1]{fontenc}
\RequirePackage{amsthm, amsmath, amsfonts, amssymb,epsf, graphicx}
\RequirePackage[square]{natbib}
\RequirePackage[colorlinks,citecolor=blue,urlcolor=blue]{hyperref}
\RequirePackage{hypernat}

\makeatletter \@addtoreset{equation}{section} \makeatother

\newcommand{\FF}{{\mathbb F}}

\newcommand{\RR}{{\mathbb  R}}
\newcommand{\SSS}{{\mathbb S}}

\newcommand{\ZZ}{{\mathbb Z}}

\newcommand{\GG}{{\mathbb G}}
\newcommand{\KK}{{\mathbb K}}

\newcommand{\argmax}{\mathop{\rm argmax}}
\newcommand{\argmin}{\mathop{\rm argmin}}

\allowdisplaybreaks[1]

\newtheorem{thm}{Theorem}
\newtheorem{lemma}{Lemma}
\newtheorem{col}{Collorary}

\newtheorem{example}{Example}

\advance \topmargin by -\headheight
\advance \topmargin by -\headsep
\advance \topmargin .2in
\begin{document}

\begin{frontmatter}
\title{A law of the iterated logarithm for Grenander's estimator}
\runtitle{LIL for Grenander}

\begin{aug}
\author{Lutz D\"umbgen,} \thanksref{t1}\ead[label=e1]{lutz.duembgen@stat.unibe.ch}
\author{Jon A. Wellner,} \thanksref{t2}\ead[label=e2]{jaw@stat.washington.edu}
\and
\author{Malcolm Wolff} \ead[label=e3]{m.lee.wolff@gmail.com}



\thankstext{t1}{Supported in part by Swiss National Science Foundation}
\thankstext{t2}{Supported in part by NSF Grant DMS-1104832 and NI-AID grant 2R01 AI291968-04} 
\runauthor{D\"umbgen, Wellner, and Wolff}



\address{University of Bern\\
       Institute of Mathematical Statistics and Actuarial Science\\
       Alpeneggstrasse 22}
\printead{e1}
\address{Department of Statistics, Box 354322\\University of Washington\\Seattle, WA  98195-4322}
\printead{e2}
\address{Department of Statistics, Box 354322\\University of Washington\\Seattle, WA  98195-4322}
\printead{e3}
\end{aug}

\begin{abstract}
In this note we prove the following law of the iterated logarithm for the 
Grenander estimator of a monotone decreasing density:
If $f(t_0) > 0$, $f'(t_0) < 0$, and $f'$ is continuous in a neighborhood of $t_0$, then
\begin{eqnarray*}
\phantom{bla}
\limsup_{n\rightarrow \infty} \left ( \frac{n}{2\log \log n} \right )^{1/3} ( \widehat{f}_n (t_0 ) - f(t_0) ) = \left| f(t_0) f'(t_0)/2 \right|^{1/3} 2M  
\end{eqnarray*}
almost surely where 
$$
M \equiv \sup_{g \in {\cal G}} T_g = (3/4)^{1/3} \ \ \ 
\mbox{and} \ \ \ 
T_g \equiv \argmax_u \{ g(u) - u^2 \} ;
$$
here ${\cal G}$ is the two-sided Strassen limit set on $\RR$.  
The proof relies on laws of the iterated logarithm for local empirical processes, 
Groeneboom's switching relation, and properties of Strassen's limit set analogous
to distributional properties of Brownian motion.
\end{abstract}



\end{frontmatter}


\bigskip

\bigskip

\section{Introduction:  the MLE of a monotone density} 
\label{sec:intro}
\hfill

\par\noindent
Nonparametric estimation of a monotone density was first considered by 
\cite{MR0086459}.  Suppose that $X_1, \ldots , X_n$ are i.i.d. with 
distribution function $F$ on $[0,\infty)$ having a decreasing density $f$.
Grenander showed that the maximum likelihood estimator
$\hat{f}_n$ of $f$ is the (left-) derivative of the least concave majorant of 
the empirical distribution function $\FF_n$
\begin{eqnarray*}
\widehat{f}_n 
& = & \{ \mbox{left derivative of the least concave majorant of }  \FF_n \}.
\end{eqnarray*}
\bigskip

\begin{figure}[htb!]
\centering
\includegraphics[width=0.5\textwidth]{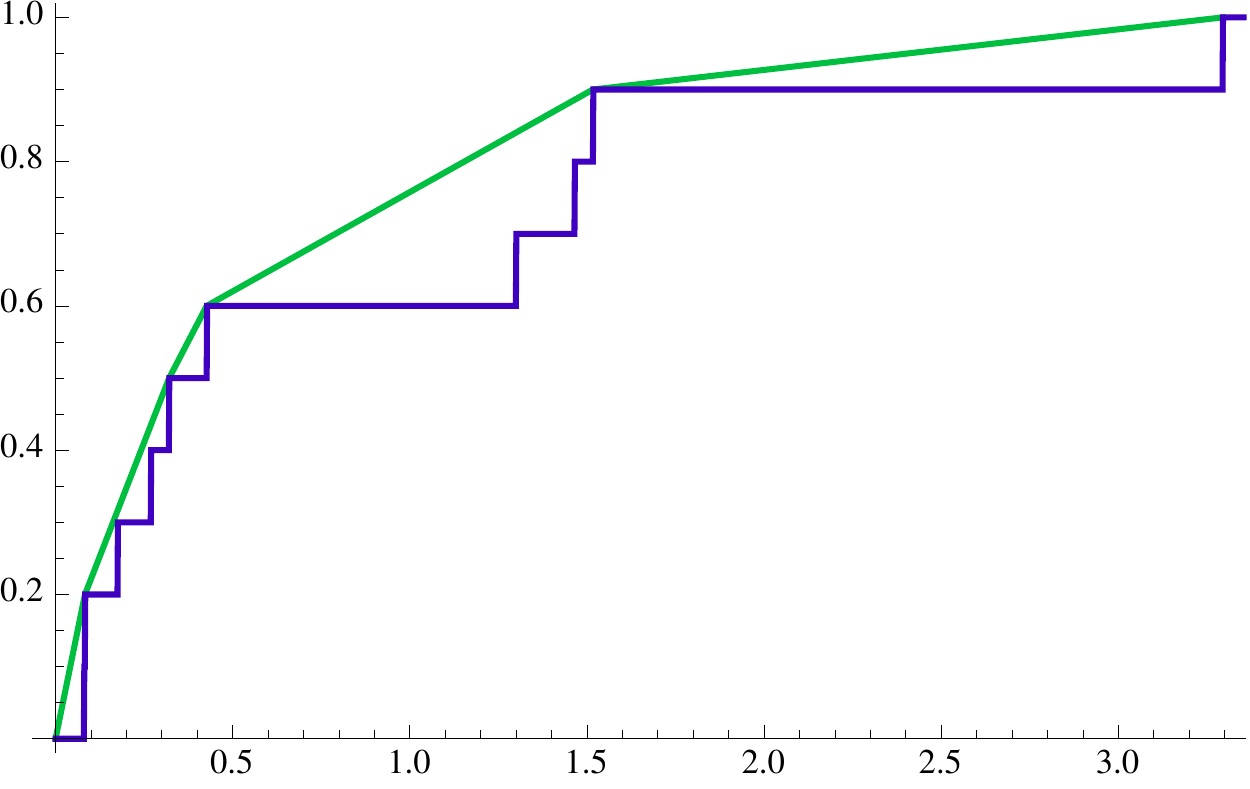}  
\caption{Empirical distribution and Least concave majorant, $n=10$}
\label{fig:figure1}
\end{figure}

\begin{figure}[htb!]
\centering
\includegraphics[width=0.5\textwidth]{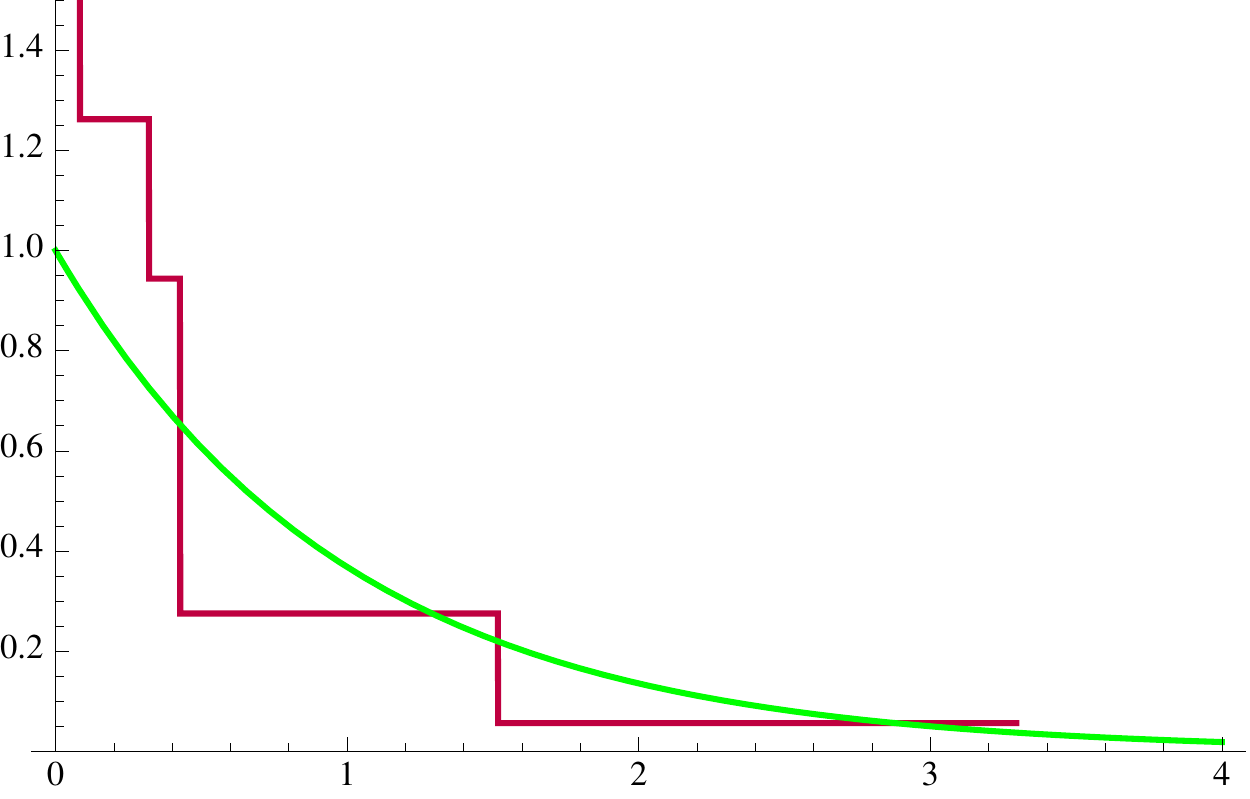}  
\caption{Grenander Estimator and Exp(1) density, $n=10$} 
\label{fig:figure2}
\end{figure}

The asymptotic distribution of $\hat{f}_n (t_0)$ at a fixed point $t_0$ with $f'(t_0) < 0$
was obtained by \cite{MR0267677}, 
and given a somewhat different proof by
\cite{MR822052}. 
If $f'(t_0) < 0$ and $f'$ is continuous in a neighborhood of $t_0$, then 
\begin{equation}
n^{1/3} (\widehat{f}_n (t_0) - f(t_0)) 
\rightarrow_d \Big | \frac{1}{2} f(t_0)f'(t_0) \Big |^{1/3} 2 \ZZ ,
\label{GrenanderLimitDistributionAtFixedPoint}
\end{equation}
where 
\begin{eqnarray}
2 \ZZ &\ = &\mbox{slope at 0 of the least concave majorant of } \ \ W(t) - t^2 \,  
\label{LimitRVFixedPoint}\\
& \stackrel{d}{=} & \mbox{slope at 0 of the greatest convex }  \ \mbox{minorant of } \ \ W(t) + t^2  \nonumber \\
& \stackrel{d}{=} & 2 \, \argmin_{t \in \RR} \{ W(t) + t^2 \}; \nonumber 
\end{eqnarray}
here $\{ W(t) : \ t \in \RR\}$ is a two-sided Brownian motion process starting 
at $0$.  
In fact, the convergence in (\ref{GrenanderLimitDistributionAtFixedPoint}) can be 
extended to weak convergence of the (local) Grenander process as follows.  
Let
$\{\SSS_{a,b} (t) : \ t \in \RR\} $ denote the slope process corresponding to the 
least concave majorant of $X_{a,b} (t) = aW(t)-bt^2$, with 
$a = \sqrt{f(t_0)}$ and $b = | f'(t_0)|/2$.
Then for fixed $t_0$ with $f'(t_0) < 0$ and 
$f'$ continuous in a neighborhood of $t_0$,
$$
n^{1/3} (\hat{f}_n (t_0 + n^{-1/3} t) - f(t_0) ) \Rightarrow  \SSS_{a,b} (t)
$$
in the Skorokhod topology on $D[-K,K]$ for every finite $K>0$; 
see e.g. 
\cite{MR981568}, 
\cite{MR1041391}, and 
\cite{MR1311975}.
\cite{MR981568} 
gives a complete analytic
characterization of the limiting distribution $\ZZ$  and further, the distributional structure of the 
process $\SSS$.  The distribution of $\ZZ = \SSS(0)/2$ has been
studied numerically by \cite{MR1939706} 
which relies heavily on 
\cite{MR822052} and \cite{MR981568}.
\cite{MR3160580} 
show that the distribution of $\ZZ$ is log-concave.
Note that there is an ``invariance principle'' involved here: the centered 
slope of the least concave majorant of $\FF_n$ converges weakly to a 
constant times the slope of the least concave majorant of $X(t) = W(t) - t^2$.
We can regard the slope in this Gaussian limit problem, $2 \ZZ$, as an 
``estimator'' of the slope of the line $2t$ in the Gaussian problem of 
``estimating'' the ``canonical'' linear function $2t$ in ``Gaussian white noise'' $dW(t)$ since
$$
d X(t) = 2 t dt + dW(t) \, .
$$

\section{A law of the iterated logarithm for the Grenander estimator} 
\label{sec:LILforGrenander}

Our main goal is to prove the following Law of the Iterated Logarithm (LIL) for the Grenander estimator
corresponding to the limiting distribution result in (\ref{GrenanderLimitDistributionAtFixedPoint}).
\medskip

\begin{thm} 
Suppose that $f(t_0) >0$, $f_0^{\prime} (t_0) < 0$ with $f^{\prime}$ continuous in a neighborhood of 
$t_0$.  Then
\begin{eqnarray*}
\limsup_{n\rightarrow\infty} \frac{n^{1/3} (  \widehat{f}_n (t_0) - f(t_0)) }{(2 \log \log n)^{1/3}} 
\ = \ \Big | \frac{1}{2} f(t_0)f'(t_0) \Big |^{1/3} 2 M
\end{eqnarray*}
almost surely where 
\begin{eqnarray*}
M \equiv \sup_{g \in {\cal G}} \argmax_{t \in \RR} \{ g(t) - t^2 \} 
=  \left ( \frac{3}{4} \right )^{1/3} ;
\end{eqnarray*}
here ${\cal G}$ is the two-sided Strassen limit set on $\RR$ given by 
\begin{equation}
{\cal G}
	\ = \ \left \{ g : \RR \rightarrow \RR \,\big|\,
	g(t) = \int_0^t \dot{g}(s) ds, \ t \in \RR,
	\ \ \int_{-\infty}^\infty \dot{g}^2 (s) ds \le 1 \right \} \, .
\label{TwoSideStrassenSetonR}
\end{equation}
\end{thm}

Our proof of Theorem 1 will rely on functional laws of the iterated logarithm for the local empirical process 
established by \cite{MR978022}; 
see also \cite{MR1303659},  
\cite{MR1652321},   
\cite{MR1440134},   
and 
\cite{MR2060302}.  
Along the way we will also prove several lemmas concerning the limit set ${\cal G}$. 
\medskip

\begin{proof}[\bf Proof]
We begin the proof of Theorem 1 with a switching argument.
Let $b_n \equiv (n^{-1}2 \log \log n)^{1/3}$.  Then we want to find a number $x_0$ such that 
\begin{eqnarray*}
P( b_n^{-1} (\widehat{f}_n (t_0 ) - f(t_0)) > x \ \ \mbox{i.o.} ) = \left \{ 
\begin{array}{l l} 0, & \ \ \mbox{if} \ x> x_0, \\ 1, & \ \ \mbox{if} \ x < x_0 . 
\end{array} \right .
\end{eqnarray*}
Now  we let
\begin{eqnarray}
\widehat{s}_n (a) \equiv \argmax_s \{ \FF_n (s) - a s \}, \ \ a \ge 0,
\label{SwitchingRelation}
\end{eqnarray} 
and note that $\{ \widehat{f}_n (t_0) > a \} = \{ \widehat{s}_n (a) > t_0 \}$
by Groeneboom's switching relation (see e.g. 
\cite{MR822052}, 
\cite{MR1385671} page 296, 
and \cite{MR2829859}, Theorem 2.1, page 881). 
Thus  the event in the last display can be rewritten as 
\begin{eqnarray}
\left \{ \widehat{f}_n (t_0) > f(t_0) + b_n x  \ \ \mbox{i.o.} \right \} 
& = & \left \{ \widehat{s}_n (f(t_0) + b_n x ) > t_0 \ \ \mbox{i.o.} \right \} .
\label{SwitchRelationforLIL}
\end{eqnarray}
But, by letting $s = t_0 + b_n h$ in (\ref{SwitchingRelation}) we see that 
\begin{eqnarray*}
\widehat{s}_n (f(t_0) + b_n x )  - t_0
& = & b_n \argmax_h \{ \FF_n (t_0 + b_n h) - (f(t_0) + b_n x)(t_0 + b_n h) \} ,
\end{eqnarray*}
and hence the right side of (\ref{SwitchRelationforLIL}) can be rewritten 
as $\{ \widehat{h}_n > 0 \ \mbox{i.o.} \}$ where
\begin{eqnarray}
\widehat{h}_n 
& = & \argmax_h \{ \FF_n (t_0 + b_n h) - (f(t_0)+ b_n x)(t_0 + b_n h) \} \nonumber \\
& = & \argmax_h \left \{ b_n^{-2} \{ \FF_n (t_0 + b_n h) - \FF_n (t_0) - (F(t_0 + b_n h) - F(t_0)) \}  \right . \nonumber \\
&& \qquad  \left . + \ b_n^{-2} \{ F(t_0 + b_n h) - F(t_0) - f(t_0)  b_n h \} - x h \right \} .  \label{ArgMaxReformulated}
\end{eqnarray}
The second term on the right side in the last display converges to $f^{\prime} (t_0) h^2/2 $ as $n \rightarrow \infty$.
The handle the first term we appeal to (a slight extension of) Theorem 2 of 
\cite{MR978022}; 
see also 
\cite{MR1303659}  
 Theorem A and Theorem 1.1, pages 1620-1621:  
by considering $h \in \RR$ and introducing the two-sided version ${\cal G}$ of the Strassen limit set 
given in (\ref{TwoSideStrassenSetonR}) much as in \cite{MR0358950},  
we see that the sequence of functions
$$
\left \{ b_n^{-2} \{ \FF_n (t_0 + b_n h) - \FF_n (t_0) - (F(t_0 + b_n h) - F(t_0))\} : \ h \in \RR \right \}
$$
is almost surely relatively compact with limit set 
$$
\{  g( f(t_0) \cdot ) : \ g \in  {\cal G} \} 
$$
where ${\cal G}$ is given by \eqref{TwoSideStrassenSetonR}.

This is most easily seen as follows:  let $\GG_n$ be the empirical d.f. of $\xi_1, \ldots , \xi_n$ i.i.d. Uniform$(0,1)$.  
As in \cite{MR1303659},  
with $n^{-1} k_n \equiv b_n$ so that $k_n = nb_n = n^{2/3} (2 \log \log n)^{1/3} \nearrow \infty$ 
and $n^{-1} k_n = b_n \searrow 0$,
the processes
\begin{eqnarray*}
\lefteqn{\frac{\xi_n (s)}{\sqrt{2 \log \log n} } } \\
& = & \frac{n^{1/2}}{\sqrt{k_n/n}} 
          \frac{\left \{ \GG_n ( F (t_0 + n^{-1} k_n s)) - \GG_n (F(t_0))  -  (F( t_0+n^{-1} k_n s) - F(t_0)) \right \}}
                 {\sqrt{2 \log \log n}}
\end{eqnarray*}
with $s\ge 0$ 
are almost surely relatively compact with limit set 
${\cal K}_{\infty} (c ) \equiv \{ t \mapsto g(ct): \ g \in {\cal K}_{\infty} \}$ with $c = f(t_0)$.
Here we also note that
\begin{eqnarray*}
\frac{n^{1/2}}{\sqrt{k_n/n} \sqrt{2 \log \log n}} = \frac{n^{2/3}}{(2 \log \log n)^{2/3}} =  b_n^{-2} .
\end{eqnarray*}
Thus the processes involved in the argmax in 
\eqref{ArgMaxReformulated} are almost surely relatively compact with limit set
\begin{eqnarray*}
\{ g( f(t_0) h ) + 2^{-1} f'(t_0) h^2 - xh : \ \ g \in {\cal G} \},
\end{eqnarray*}
and by Lemma 1 below this set is equal to 
\begin{eqnarray*}
\left \{ a g(h) - b h^2 - xh : \ \ g \in {\cal G} \right \}
\end{eqnarray*}
where $a \equiv \sqrt{ f(t_0)}$, and $b = |f'(t_0)|/2$.
Thus by Lemma 2 below, the set of limits for the argmax in \eqref{ArgMaxReformulated} equals 
\begin{eqnarray*}
\left \{ (a/b)^{2/3} \argmax_h \{ g(h) - h^2 \} - x/(2b) : \ \ g \in {\cal G} \right \}
\end{eqnarray*}
where
\begin{eqnarray*}
\left ( \frac{a}{b} \right )^{2/3} = \left ( \frac{\sqrt{ f(t_0)}}{2^{-1} | f'(t_0) |} \right )^{2/3} =  \left ( \frac{4 f(t_0)}{|f'(t_0)|^2} \right )^{1/3} .
\end{eqnarray*}
Hence, with $T_g = \argmax_h \{ g(h) - h^2 \}$, 
\begin{eqnarray*}
\left \{ \widehat{h}_n > 0 \ \ \mbox{i.o.} \right \} 
& \stackrel{a.s.}{=} & \left \{ \left ( \frac{a}{b} \right )^{2/3} \sup_{g \in {\cal G}} T_g > \frac{x}{2b} \right \} \\
& = & \left \{ 2 b \left ( \frac{a}{b} \right )^{2/3} \sup_{g \in {\cal G}} T_g > x \right \} \\
& = & \emptyset 
\end{eqnarray*}
if 
$$
x > x_0 \equiv 2 b \left ( \frac{a}{b} \right )^{2/3} \sup_{g \in {\cal G}} T_g 
= \Big | \frac{1}{2} f(t_0)f'(t_0) \Big |^{1/3} 2 \sup_{g \in {\cal G}} T_g .
$$
It remains only to show that 
$\sup_{g\in {\cal G}} T_g = (3/4)^{1/3}$.  This follows from Lemma 3 in Section 4 below.
\end{proof}
 
\begin{lemma} 
Let $c > 0$ and $d \in \RR$. Then 
$$
	\left \{ t \mapsto g(ct+d) - g(d) : \ g \in {\cal G} \right \}
	= \sqrt{c}{\cal G} .
$$
\end{lemma}

\begin{proof}[\bf Proof]
If $g \in {\cal G}$, then 
\begin{eqnarray*}
g(ct+d) - g(d) 
& = & \int_d^{ct+d} \dot{g}(s) ds = \int_0^{ct} \dot{g} (v + d) dv = \int_0^t \dot{g} (cu +d ) c du \\
& = & \sqrt{c} \int_0^t \sqrt{c} \dot{g} (cu + d) du \\
& = & \sqrt{c} \tilde{g} (t) 
\end{eqnarray*}
where $\tilde{g} \in {\cal G}$ since 
\[
\int_{-\infty}^\infty ( \sqrt{c} \dot{g} (cu+d))^2 du = \int_{-\infty}^{\infty} \dot{g}^2 (w) dw \le 1 .
\]
This shows that the set of functions $t \mapsto g(ct + d) - g(d)$, $g \in {\cal G}$, is contained in $\sqrt{c} {\cal G}$. On the other hand, any function $\tilde{g} \in {\cal G}$ with derivative $\dot{\tilde{g}}$ may be written as $\tilde{g}(t) = \int \sqrt{c} \dot{g}(cu + d) du$ with $\dot{g}$ given by $\dot{g}(s) \equiv \sqrt{c^{-1}} \dot{\tilde{g}}(c^{-1}s - c^{-1}d)$ and satisfying $\int_{-\infty}^\infty \dot{g}(s)^2 ds = \int_{-\infty}^\infty \dot{\tilde{g}}(s)^2 ds \le 1$.
\end{proof}

\begin{lemma} 
Let $\alpha,\beta$ be positive constants and $\gamma \in \RR$. Then 
\begin{eqnarray}
\lefteqn{\left \{ \argmax_h \{ \alpha g(h) - \beta h^2 - \gamma h \} : \ \ g \in {\cal G} \right \} }\nonumber \\
& = & 
\left \{ (\alpha/\beta)^{2/3} \argmax_h \{ g(h) - h^2 \} - \gamma/(2\beta) : \ \ g \in {\cal G} \right \} .
\label{StrassenArgMaxIdentity}
\end{eqnarray}
\end{lemma}

\begin{proof}[\bf Proof]
Note first that
\begin{eqnarray*}
	M_g
	& \equiv & \argmax_h \left\{ \alpha g(h) - \beta h^2 - \gamma h \right\} \\
	& = & \argmax_h \left\{ \alpha g(h) - \beta (h + \gamma/(2\beta))^2 \right\} \\
	& = & \argmax_h \left\{ g(h) - (\beta/\alpha) (h + \gamma/(2\beta))^2 \right\} \\
	& = & \argmax_v \left\{ g(v + d) - (\beta/\alpha) v^2 \right\} + d
\end{eqnarray*}
with $d := - \gamma/(2\beta)$. Moreover, for any $c > 0$ and
\[
	\tilde{g}(u) \equiv c^{-1/2} \left( g(cu + d) - g(d) \right)
\]
we may write
\begin{eqnarray*}
	M_g
	& = & c \argmax_u \left\{ g(cu + d) - g(d) - (\beta/\alpha) c^2 u^2 \right\}
		+ d \\
	& = & c \argmax_u \left\{ c^{1/2} \tilde{g}(u) - (\beta/\alpha) c^2 u^2 \right\}
		+ d \\
	& = & c \argmax_u \left\{ \tilde{g}(u) - (\beta/\alpha) c^{3/2} u^2 \right\}
		+ d .
\end{eqnarray*}
In case of $c = (\alpha/\beta)^{2/3}$ we obtain
\[
	M_g = (\alpha/\beta)^{2/3} \argmax_u \left\{ \tilde{g}(u) - u^2 \right\}
		- \gamma / (2\beta) .
\]
Now the claim follows from Lemma~1, because the set $\left\{ \tilde{g} : g \in {\cal G} \right\}$ equals ${\cal G}$.
\end{proof}


\section{Some comparisons and connections}
\label{sec:Comparisons}

As noted in the introduction, 
$$
2 \ZZ \stackrel{d}{=} \mbox{slope at zero of the least concave majorant of} \ \ W(t) - t^2 .
$$
This suggests that with $T_g = \mbox{argmax}_t \{ g(t) - t^2 \}$ we have
\begin{eqnarray*}
\lefteqn{\left \{ 2 \sup T_{g} : \ g \in {\cal G} \right \} }\\
& = & \sup 
\{ \mbox{slope at} \ 0 \ \mbox{of the least concave majorant of } \ g(t) - t^2 : \ g \in {\cal G} \} .
\end{eqnarray*}

\section{Proof for the variational problem}
\label{sec:ProofVariationalProblem}

It is natural to conjecture that $\sup_{g \in {\cal G}} T_g = (3/4)^{1/3} \approx 0.90856 \ldots$.
This is motivated by the asymptotic behavior of Chernoff's density; see 
\cite{MR981568}, Corollary 3.4, page 94:  since the density 
$$
f_{\ZZ} (z) \sim \frac{1}{2Ai^{\prime} (a_1) } 4^{4/3} z \exp \left ( - \frac{2}{3} z^3 + 3^{1/3} a_1 z \right )
$$
as $z \rightarrow \infty$, the tail probability $P( \ZZ >z )$ satisfies
$$
P( \ZZ > z ) \sim \frac{1}{2Ai^{\prime} (a_1) } 4^{4/3}  \frac{1}{z} \exp \left ( - \frac{2}{3} z^3 \right ) 
$$
as $z\rightarrow \infty$ where
$a_1 \dot= -2.3381$ is the largest zero of the Airy function $Ai$ and 
$Ai^{\prime} (a_1) \dot= 0.7022$.
  Thus from (\ref{GrenanderLimitDistributionAtFixedPoint}) we expect that
\begin{eqnarray*}
\limsup_{n\rightarrow\infty} \frac{n^{1/3} (  \widehat{f}_n (t_0) - f(t_0)) }{((3/2) \log \log n)^{1/3}} 
\ = \ \Big | \frac{1}{2} f(t_0)f'(t_0) \Big |^{1/3} 2 ,
\end{eqnarray*}
or, equivalently,
\begin{eqnarray*}
\limsup_{n\rightarrow\infty} \frac{n^{1/3} (  \widehat{f}_n (t_0) - f(t_0)) }{(2 \log \log n)^{1/3}} 
& \  = \ & \Big | \frac{1}{2} f(t_0)f'(t_0) \Big |^{1/3} 2 \cdot \frac{1}{2^{1/3}} \cdot \left ( \frac{3}{2} \right )^{1/3} \\
& \ = \ &  \Big | \frac{1}{2} f(t_0)f'(t_0) \Big |^{1/3} 2 \cdot  \left ( \frac{3}{4} \right )^{1/3} .
\end{eqnarray*}
On the other hand the proof of Theorem 1 above leads to
\begin{eqnarray*}
\limsup_{n\rightarrow\infty} \frac{n^{1/3} ( \widehat{f}_n (t_0) - f(t_0)) }{(2 \log \log n)^{1/3}} 
& \  = \ & \Big | \frac{1}{2} f(t_0)f'(t_0) \Big |^{1/3} 2 \cdot M\ \ \ \mbox{a.s.}
\end{eqnarray*}
where 
$$
M \equiv \sup_{g \in {\cal G}} \mbox{argmax}_{t \in \RR} \{ g(t) - t^2 \} \equiv \sup_{g \in {\cal G}} T_g .
$$
Thus we conjecture that $M = (3/4)^{1/3} $.
\medskip

\begin{lemma}
Let $t_0 > 0$ be an arbitrary positive number and let $\dot{g} \in L_1 ([0,t_0])$ 
be an arbitrary function satisfying
\begin{eqnarray*}
	\int_0^{t_0} \dot{g} (s) ds - t_0^2 \ge \int_0^t \dot{g}(s) ds - t^2
	\ \ \ \mbox{for} \ \ 0 \le t \le t_0 .
\end{eqnarray*}
Then
\begin{eqnarray*}
	\int_0^{t_0} \dot{g} (u)^2 du \ge \int_0^{t_0} (2u)^2 du
	= \frac{4t_0^3}{3} .
\end{eqnarray*}
\end{lemma}

\begin{proof}[\bf Proof]
Let $\dot{g}_0 (u) \equiv 2u$.  The claimed inequality is trivial if the 
integral on the left side is infinite, so we may
view $\dot{g}$ and $\dot{g_0}$ as elements of the Hilbert space 
$L_2 ([0,t_0])$.  Then the assumption on $\dot{g}$ may be rewritten as
\begin{eqnarray*}
\langle \dot{g} - \dot{g}_0 , 1 \rangle \ge \langle \dot{g} - \dot{g}_0 , 1_{[0,t]} \rangle \ \ \mbox{for}  \ \ 0 \le t \le t_0 .
\end{eqnarray*}
In other words,
\begin{eqnarray*}
\langle \dot{g} - \dot{g}_0 , 1_{(t,t_0]} \rangle \ge 0 \ \ \ \mbox{for} \ \ 0 \le t \le t_0 ,
\end{eqnarray*}
and this is equivalent to 
\begin{eqnarray*}
\langle \dot{g} - \dot{g}_0 , f \rangle \ge 0 
\end{eqnarray*}
for all functions $f$ in the closed convex cone $\KK$ generated by the indicator functions $1_{(t,t_0]}$.  
This is the set of non-negative and non-decreasing functions on $[0,t_0]$.  In particular, 
$\dot{g}_0 \in \KK$, so 
\begin{eqnarray*}
\langle \dot{g} - \dot{g}_0 , \dot{g}_0 \rangle \ge 0 .
\end{eqnarray*}
Together with the Cauchy-Schwarz inequality we obtain 
\begin{eqnarray*}
0 \le \langle \dot{g} - \dot{g}_0, \dot{g}_0 \rangle 
= \langle \dot{g} , \dot{g}_0 \rangle - \| \dot{g}_0 \|^2  \le \| \dot{g} \| \| \dot{g}_0 \| -  \| \dot{g}_0 \|^2 ,
\end{eqnarray*}
so $\| \dot{g} \| \ge \|\dot{g}_0 \|$.  This inequality is strict unless $\dot{g} = \lambda \dot{g}_0$ 
for some $\lambda \in \RR$.  In this special case the last display reads 
$0 \le ( \lambda - 1) \| \dot{g}_0 \|^2$, so $\lambda \ge 1$ and $\| \dot{g}\| = \lambda \| \dot{g}_0 \|$
with equality if, and only if, $\lambda =1$ and $\dot{g} = \dot{g}_0 $.
\end{proof}


\begin{example} 
If we take $f(x) = e^{-x} 1_{[0,\infty)} (x)$ and $t_0 = \log 2$, then 
\begin{eqnarray*}
\Big | \frac{1}{2} f(t_0)f'(t_0) \Big |^{1/3} \cdot 2 = (2^{-3})^{1/3} \cdot 2 = 1,
\end{eqnarray*}
so the limit superior is  just $\sup_{g \in {\cal G}} T_g = (3/4)^{1/3}$.
\end{example}

\begin{example} 
If we take $f(x) = (1+x)^{-2} 1_{[0,\infty)} (x)$, then $-f^{\prime}(x) = 2(1+x)^{-3}$ and hence
with $t_0 =1$ we have $f(1) = 1/4 = - f^{\prime} (1)$.  Then
\begin{eqnarray*}
\Big | \frac{1}{2} f(t_0)f'(t_0) \Big |^{1/3} \cdot 2 = (2^{-5/3}) \cdot 2 = 2^{-2/3},
\end{eqnarray*}
so the limit superior is  $2^{-2/3} \sup_{g \in {\cal G}} T_g = (3/16)^{1/3}$.
\end{example}

\begin{example} 
If we take $f(x) = (\sqrt{2} -x)1_{[0,\sqrt{2}]} (x)$ and $t_0 = \sqrt{2}-1$, 
then $f(t_0) = 1$, $- f^{\prime} (t_0 ) = 1$, and 
\begin{eqnarray*}
\Big | \frac{1}{2} f(t_0)f'(t_0) \Big |^{1/3} \cdot 2 = (2^{-1/3}) \cdot 2 = 2^{+2/3},
\end{eqnarray*}
so the limit superior is  $2^{+2/3} \sup_{g \in {\cal G}} T_g = 2^{2/3} (3/4)^{1/3} = 3^{1/3}$.
\end{example}

\section{Some corollaries}

Theorem~1 has a number of corollaries and consequences, since the argument in the proof applies to a number 
of problems involving nonparametric estimation of a monotone function.   
Our first corollary, however, involves estimation of the mixing distribution $G$ in the mixture representation of
a monotone density:  
that is, 
\begin{eqnarray}
f(x) 
 =  \int_0^{\infty} \frac{1}{y} 1_{[0,y)} (x) d G(y)         
 =  \int_{\{y>x\} } \frac{1}{y} dG(y) , 
\qquad x \in (0,\infty) 
 \label{MixRepresA-Monotone} 
\end{eqnarray}
for some distribution function $G$ on $(0,\infty)$.
This fact apparently goes back at least to \cite{schoenberg:41};
see the introduction of \cite{MR0077581}, and \cite{MR0270403}, page 158. 
The relationship (\ref{MixRepresA-Monotone}) implies that the corresponding 
distribution function $F$ is given by
\begin{eqnarray*} 
F(x) & = & \int_0^{\infty} \frac{x}{y} 1_{[0,y)} (x) dG(y) + \int_0^{\infty} 1_{[y,\infty)} (x) dG(y) \\
& = & x f(x) + G(x) \, ,
\end{eqnarray*}
and this can be ``inverted'' to yield
\begin{equation}
G(x) = F(x) - x f(x) \, .
\label{InverseOfAMonotone}
\end{equation}
From Figure~\ref{fig:figure3}  we see that the function on the 
right side of (\ref{InverseOfAMonotone}) 
is non-negative and non-decreasing:  
the shaded area gives exactly the difference $F(x) - xf(x)$.

\begin{figure}[htb!]
\centering
\includegraphics[width=0.5\textwidth]{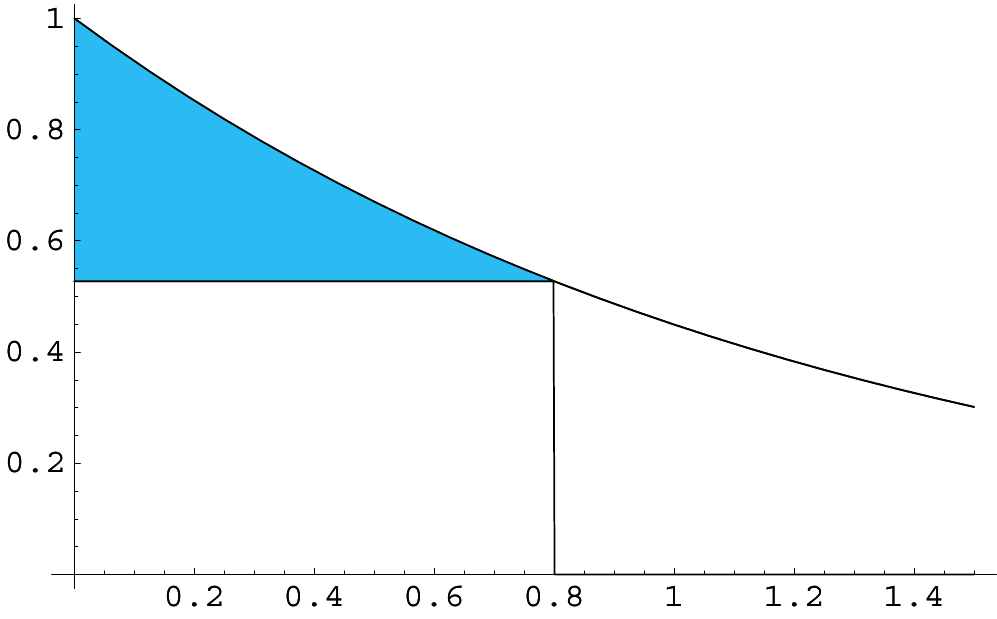}  
\caption{Graphical view of the inversion formula,  monontone density}  
\label{fig:figure3}
\end{figure}

The identity (\ref{InverseOfAMonotone}) implies that the nonparametric maximum 
likelihood estimator of $G$ is $\widehat{G}_n$ given by
\begin{eqnarray*}
\widehat{G}_n (t) = \widehat{F}_n (t) - t \widehat{f}_n (t) , \ \ \mbox{for} \ \ t\ge 0 
\end{eqnarray*}
where $\widehat{F}_n (t) = \int_0^{t} \widehat{f}_n (x)dx $ is the least concave majorant of $\FF_n$ 
and the MLE of $F$ assuming that $f$ is monotone (and hence $F$ is concave).
Thus for $t_0 > 0$ we can write 
\begin{eqnarray*}
n^{1/3} (\widehat{G}_n (t_0) - G(t_0)) = n^{1/3} ( \widehat{F}_n (t_0) - F(t_0)) - t_0 n^{1/3} (\widehat{f}_n (t_0) - f(t_0)) 
\end{eqnarray*}
From Marshall's lemma \cite{MR0273755} and $n^{1/2} \| \FF_n - F \|_{\infty} = O_p (1)$ it follows that 
$n^{1/3} \| \widehat{F}_n - F \|_{\infty} = o_p (1)$.  Thus if $t_0>0$ is a point at which the hypotheses of Theorem 1 hold,
then the convergence in (\ref{GrenanderLimitDistributionAtFixedPoint}) implies that 
\begin{equation}
n^{1/3} (\widehat{G}_n (t_0) - G(t_0)) 
\rightarrow_d t_0 \Big | \frac{1}{2} f(t_0)f'(t_0) \Big |^{1/3} 2 \ZZ ,
\label{MixingDFMLELimitDistributionAtFixedPoint}
\end{equation}
Similarly, from Marshall's lemma \cite{MR0273755} and Chung's law of the iterated logarithm for 
$\| \FF_n - F\|_{\infty}$ 
(see e.g. \cite{MR838963}, page 505), we know that with $b_n \equiv (2 \log \log n)^{1/2}$
$$
\limsup_{n\rightarrow \infty} n^{1/2} \| \widehat{F}_n - F \|_{\infty}/b_n  \le 
\limsup_{n\rightarrow \infty} n^{1/2} \| \FF_n - F \|_{\infty}/b_n  = 1/2 \ \ \mbox{a.s.} .
$$ 
It follows that if $t_0>0$ is a point at which the hypotheses of Theorem 1 hold, then Theorem 1 yields a LIL result for 
$\widehat{G}_n (t_0)$ as follows:
\medskip

\begin{col}
Suppose that  $f(t_0) > 0$ and $f'(t_0) < 0$ with $f'$ continuous in a neighborhood of $t_0$.
Then 
\begin{eqnarray*}
\limsup_{n\rightarrow \infty} \frac{n^{1/3} ( \widehat{G}_n (t_0) - G(t_0))}{(2 \log \log n)^{1/3}} 
= t_0 \Big | \frac{1}{2} f(t_0)f'(t_0) \Big |^{1/3} 2 (3/4)^{1/3} \
\end{eqnarray*}
almost surely.
\end{col}

\section{A further problem} 
For the problem of estimating a convex decreasing density,
\cite{MR1891741} described the limiting distribution of the estimator 
(at a point under a natural curvature condition) in terms of an ``invelope'' of two-sided 
integrated Brownian motion plus $t^4$ which was characterized in \cite{MR1891741}.
The same distribution has appeared in other nonparametric convex function estimation problems, 
for example for log-concave density estimation:  see \cite{MR2509075}.
In spite of this description of the limiting distribution for the convex density case
 in terms of integrated Brownian motion, almost
nothing is known concerning a direct analytical description of the limit distribution comparable 
to the results of \cite{MR822052,MR981568} for Chernoff's distribution.  
(On the other hand, a preliminary numerical investigation of the distribution is given by \cite{MR3178372}.)

This leads to the following question:  can some information concerning 
the constants involved in the limiting distribution in the convex function case 
be obtained by establishing LIL results analogous 
to those established here in the monotone case?

\section*{Acknowledgements}  
The second author owes thanks to Piet Groeneboom for many discussions concerning the Grenander 
estimator and monotone function estimation more generally.  
Thanks are due as well to David Mason for references concerning local LIL's for empirical processes.
$\phantom{blabla}$  


\end{document}